\documentclass[11pt]{amsart}

\usepackage{amsmath}
\usepackage{fullpage}
\usepackage{xspace}
\usepackage[psamsfonts]{amssymb}
\usepackage[spanish]{babel}
\usepackage[utf8]{inputenc}
\usepackage{graphicx,color}
\usepackage[curve]{xypic} 
\usepackage{hyperref}
\usepackage{graphicx}
\usepackage{amsmath}%
\usepackage{amsthm}%
\usepackage{amscd}
\usepackage{amsfonts}%
\usepackage{amssymb}%
\usepackage{graphicx}

\usepackage{mathrsfs}

\usepackage{tikz}
\usetikzlibrary{matrix,arrows}

\newtheorem{theorem}{Teorema}[section]

\newtheorem{conjecture}[theorem]{Conjetura}

\newtheorem{lemma}[theorem]{Lema}

\theoremstyle{remark}

\numberwithin{equation}{section}

\newcommand{\Ncal}{\mathscr{N}}

\newcommand{\Scal}{\mathscr{S}}

\newcommand{\Pro}{\mathbb{P}}

\newcommand{\Z}{\mathbb{Z}}
\newcommand{\C}{\mathbb{C}}

\newcommand{\Q}{\mathbb{Q}}

\newcommand{\rk}{\mathrm{rg}\,}

\newcommand{\MW}{\mathrm{MW}}

  \DeclareFontFamily{U}{wncy}{}
    \DeclareFontShape{U}{wncy}{m}{n}{<->wncyr10}{}
    \DeclareSymbolFont{mcy}{U}{wncy}{m}{n}
    \DeclareMathSymbol{\Sha}{\mathord}{mcy}{"58}

\begin{document}
\title[Superficies elípticas y el décimo problema de Hilbert]{Superficies elípticas y el décimo problema de Hilbert}

\author{Héctor Pastén}
\address{ Departamento de Matem\'aticas\newline
\indent Pontificia Universidad Cat\'olica de Chile, Facultad de Matem\'aticas\newline
\indent 4860 Av. Vicu\~na Mackenna,  Macul, RM, Chile}
\email[H. Pastén]{hector.pasten@mat.uc.cl}%

%\thanks{}
\thanks{Financiado por ANID (ex CONICYT) FONDECYT Regular 1190442 de Chile.}
\date{\today}
\subjclass[2010]{Primario: 11U05; Secundario: 14J27, 11G05} %
\keywords{Décimo problema de Hilbert, anillos de enteros, superficies elípticas, curvas elípticas}%
%\dedicatory{}

\begin{abstract}  Es sabido que se obtendría una solución negativa al décimo problema de Hilbert para el anillo de enteros $O_F$ de un campo de números $F$ si $\mathbb{Z}$ fuera diofantino en $O_F$. Denef y Lipshitz conjeturaron que esto último ocurre para todo $F$. En esta nota se demuestra que la conjetura de Denef y Lipshitz es consecuencia de una conocida conjetura sobre superficies elípticas.
\end{abstract}

% A negative answer to Hilbert's tenth problem for the ring of integers $O_F$ of a number field $F$ would follow if $\mathbb{Z}$ were Diophantine in $O_F$. Denef and Lipshitz conjectured that the latter occurs for every number field $F$. In this note we show that the conjecture of Denef and Lipshitz is a consequence of a well-known conjecture on elliptic surfaces.

\maketitle

%\tableofcontents

%%%%%%%%%%%%%%%%%%%%%%%%%%%%%%%%%%%%%%
%%%%%%%%%%%%%%%%%%%%%%%%%%%%%%%%%%%%%%
%%%%%%%%%%%%%%%%%%%%%%%%%%%%%%%%%%%%%%
%%%%%%%%%%%%%%%%%%%%%%%%%%%%%%%%%%%%%%
%%%%%%%%%%%%%%%%%%%%%%%%%%%%%%%%%%%%%%
%%%%%%%%%%%%%%%%%%%%%%%%%%%%%%%%%%%%%%

\section{Introduction}

 A partir de la solución negativa al décimo problema de Hilbert \cite{DPR, Matiyasevich}  se han investigado varias extensiones a otras estructuras. Uno de los casos abiertos más destacados es el de anillos de enteros de campos de números. En este contexto el problema se ha resuelto negativamente para el anillo de enteros $O_F$ de varios tipos de campos de números $F$, ver \cite{DenefQuad, DenLip, Denef,   Pheidas, ShaShl, ShlH10, Videla, MazRubDS, GFP, Ray, KLS}. Sin embargo, el caso general sigue abierto y se espera una respuesta negativa. En efecto, si $\Z$ es diofantino en $O_F$ (ver la definición de ``conjunto diofantino'' en la sección siguiente) entonces el análogo del décimo problema de Hilbert en $O_F$ tiene respuesta negativa, mientras que Denef y Lipshitz \cite{DenLip} propusieron:

 \begin{conjecture}[Denef--Lipshitz] El anillo $\Z$ es diofantino en $O_F$ para todo campo de números $F$.
\end{conjecture}

Aunque la conjetura de Denef y Lipshitz sigue abierta en general, los trabajos de Mazur--Rubin  \cite{MazRubH10} y de Murty--Pasten \cite{MurtyPasten} muestran que es consecuencia de conjeturas estándar sobre curvas elípticas. En esta nota se dará evidencia adicional, demostrando que la conjetura de Denef y Lipshitz también es consecuencia de una conocida conjetura sobre superficies elípticas.

 Dado un campo de números $K$, sea $\pi:X\to \Pro^1$ una superficie elíptica sobre $K$ con una sección distinguida $\sigma_0:\Pro^1\to X$. Por \cite{LangNeron} sabemos que el grupo de secciones $\MW(X,\pi,K)$ definidas sobre $K$ con neutro $\sigma_0$ es finitamente generado. Para todos salvo finitos $t\in \Pro^1(K)$ la fibra $X_{t}$ es una curva elíptica y se tiene el morfismo de especialización ${\rm esp}_t:\MW(X,\pi,K)\to X_t(K)$ dado por $\sigma\mapsto \sigma(t)$. Mejorando resultados de \cite{NeronSp}, Silverman \cite{SilvermanSp} demostró lo siguiente ($\rk$ es el rango):

\begin{theorem}[Silverman] \label{ThmSilvermanSp}
El morfismo de especialización ${\rm esp}_t$ es inyectivo para todos salvo finitos  $t$ en $\Pro^1(K)$. Así, para todos salvo finitos  $t$ en $\Pro^1(K)$ se tiene que $\rk X_t(K)\ge \rk \MW(X,\pi, K)$.
\end{theorem} 
Así, es difícil no definir los conjuntos de puntos $t\in \Pro^1(K) $ donde $\rk X_t(K)=\rk \MW(X,\pi,K)$ y donde $\rk X_t(K)>\rk \MW(X,\pi,K)$, los cuales denotaremos por $\Ncal(X,\pi,K)$ y $\Scal(X,\pi,K)$ respectivamente; esto es, donde el rango no salta y donde el rango salta. Es una conjetura que forma parte del folklore del área que si $\pi:X\to \Pro^1$ no es isotrivial, entonces ambos conjuntos deberían ser infinitos. Este problema ha capturado considerable atención, ver por ejemplo \cite{Salgado} y las referencias ahí citadas. Nuestra contribución es:

\begin{theorem}\label{ThmMain2} Suponga que para todo campo de números $K$ y toda superficie elíptica no isotrivial $\pi:X\to \Pro^1$ sobre $K$ con $\rk \MW(X,\pi,K)=0$ se tiene que $\Ncal(X,\pi,K)$ es infinito. Entonces la conjetura de Denef y Lipshitz es correcta: $\Z$ es diofantino en $O_F$ para todo campo de números $F$.
\end{theorem}

%%%%%%%%%%%%%%%%%%%%%%%%%%%%%%%%%%%%%%
%%%%%%%%%%%%%%%%%%%%%%%%%%%%%%%%%%%%%%
%%%%%%%%%%%%%%%%%%%%%%%%%%%%%%%%%%%%%%
%%%%%%%%%%%%%%%%%%%%%%%%%%%%%%%%%%%%%%
%%%%%%%%%%%%%%%%%%%%%%%%%%%%%%%%%%%%%%
%%%%%%%%%%%%%%%%%%%%%%%%%%%%%%%%%%%%%%

\section{Preliminares}

Dado un campo de números $F$, un conjunto $S\subseteq O_F$ es \emph{diofantino} si existe un polinomio $P(x,y_1,...,y_n)\in O_F[x,y_1,...,y_n]$ (para algún $n$) tal que dado cualquier $a\in O_F$, se tiene que $a\in S$ si y solo si la ecuación $P(a,y_1,...,y_n)=0$ tiene solución sobre $O_F$. Por ejemplo, los cuadrados de $O_F$ forman un conjunto diofantino, tomando $P(x,y_1)=x-y_1^2$.

En un trabajo pionero, Denef \cite{Denef} relacionó conjuntos diofantinos en anillos de enteros a estabilidad de rangos de curvas elípticas en extensiones. Este tema fue desarrollado por Poonen \cite{Poonen},  Cornelissen--Pheidas--Zahidi \cite{CPZ}, y mas recientemente por Shlapentokh \cite{ShlE}, quien obtuvo (simultáneamente con Poonen) el teorema siguiente:

\begin{theorem}[Poonen, Shlapentokh]\label{ThmEllCriterion} Sea $L/K$ una extensión de campos de números. Si hay una curva elíptica $E$ sobre $K$ con $\rk E(L)=\rk E(K)>0$, entonces $O_K$ es diofantino en $O_L$.
\end{theorem}

El resultado anterior ha jugado un papel central en todos los últimos desarrollos en la conjetura de Denef y Lipshitz. 

Más aún, usando \cite{Denef}, Shlapentokh probó la siguiente sorprendente reducción \cite{MRS}:

\begin{lemma}[Shlapentokh]\label{LemmaShlQuad} Suponga que para toda extensión cuadrática de campos de números $L/K$ se tiene que $O_K$ es diofantino en $O_L$. Entonces la conjetura de Denef y Lipshitz es correcta.
\end{lemma}

Además de lo anterior, necesitaremos un ejemplo conveniente de superficie elíptica.

\begin{lemma}\label{LemmaES} Existe una superficie elíptica no isotrivial $\pi:Y\to \Pro^1$ definida sobre $\Q$ tal que el grupo de secciones sobre $\C$ es exactamente  $\MW(Y,\pi, \Q)$, y este grupo es libre abeliano de rango positivo.
\end{lemma}

Superficies elípticas de este tipo abundan, ver por ejemplo \cite{Schwartz}.

%%%%%%%%%%%%%%%%%%%%%%%%%%%%%%%%%%%%%%
%%%%%%%%%%%%%%%%%%%%%%%%%%%%%%%%%%%%%%
%%%%%%%%%%%%%%%%%%%%%%%%%%%%%%%%%%%%%%
%%%%%%%%%%%%%%%%%%%%%%%%%%%%%%%%%%%%%%
%%%%%%%%%%%%%%%%%%%%%%%%%%%%%%%%%%%%%%
%%%%%%%%%%%%%%%%%%%%%%%%%%%%%%%%%%%%%%

\section{La conjetura de Denef y Lipshitz cuando el rango no salta} \label{SecNJ}

Sea $\pi:Y\to \Pro^1$ una superficie elíptica sobre $\Q$ como las dadas por el Lema \ref{LemmaES}. Sea $y^2=f(T, x)$ una ecuación de Weierstrass para ella con coeficientes en $\Q[T]$ donde $T$ es una coordenada afín en $\Pro^1$. Dado un campo de números $K$ y un $\beta\in K^\times$ definimos la superficie elíptica $p: Y_K^{(\beta)}\to \Pro^1$ sobre $K$ por medio de la ecuación $\beta  y^2= f(T, x)$, dotada de la sección $\sigma_0$ al infinito. Sobre $\C$ esta superficie elíptica es isomorfa a $\pi: Y\to \Pro^1$, así que no es isotrivial.

\begin{lemma}\label{LemmaRkBeta} Si $K$ es un campo de números y $\beta\in K$ no es un cuadrado en $K$, entonces se tiene que $\MW(Y_K^{(\beta)},p,K)=(\sigma_0)$, o sea, este grupo de secciones sobre $K$ es trivial.
\end{lemma}
\begin{proof} Suponga que $\tau=(x,y)=(a,b)\ne \sigma_0$ es una sección con $a,b\in K(T)$. Tenemos $b\ne 0$ porque sobre $\C$ no hay secciones no triviales de $2$-torsión (por el Lema \ref{LemmaES}). Así, si $\gamma$ es una raíz cuadrada de $\beta$ vemos que $\tau'=(a,\gamma b)$ es una sección de $\pi:Y\to \Pro^1$ sobre  $\C$, y por ende sobre $\Q$ por el Lema \ref{LemmaES}. Luego, $\gamma b\in \Q(T)$ y como $b\in K(T)^\times$ concluimos que $\gamma\in K$;  contradicción.
\end{proof}
\begin{proof}[Demostración del Teorema \ref{ThmMain2}] Sea $L/K$ cualquier extensión cuadrática de campos de números. Es generada por la raíz cuadrada de cierto $\beta\in K^\times$ que no es cuadrado. El Lema \ref{LemmaRkBeta} da que $\rk \MW(Y_K^{(\beta)},p,K)=0$. Por hipótesis $\Ncal(Y_K^{(\beta)},p,K)$ es infinito, así que hay infinitos $\alpha\in K$ con $\rk E_\alpha^{(\beta)}(K)=0$, donde $E_\alpha^{(\beta)}$ es la curva elíptica dada por $\beta y^2=f(\alpha,x)$. Esta última es el torcimiento cuadrático por $L$ sobre $K$ de la curva elíptica $E_\alpha$ definida por $y^2=f(\alpha,x)$.

Como $\rk \MW(Y,\pi, K)\ge 1$ (ver el Lema \ref{LemmaES}), del Teorema \ref{ThmSilvermanSp} obtenemos que todos salvo finitos $\alpha\in K$ cumplen que $\rk E_\alpha(K)\ge 1$. Así, podemos elegir $\alpha\in K$ con $\rk E_\alpha^{(\beta)}(K)=0$ y $\rk E_\alpha(K)\ge 1$. Hecho esto, se obtiene
$$
\rk E_\alpha(L) =  \rk E_\alpha(K) + \rk E^{(\beta)}_\alpha(K)=\rk E_\alpha(K)>0
$$
y por el Teorema \ref{ThmEllCriterion} deducimos que $O_K$ es diofantino en $O_L$. Concluimos por el Lema \ref{LemmaShlQuad}.
\end{proof}

%%%%%%%%%%%%%%%%%%%%%%%%%%%%%%%%%%%%%%
%%%%%%%%%%%%%%%%%%%%%%%%%%%%%%%%%%%%%%
%%%%%%%%%%%%%%%%%%%%%%%%%%%%%%%%%%%%%%
%%%%%%%%%%%%%%%%%%%%%%%%%%%%%%%%%%%%%%
%%%%%%%%%%%%%%%%%%%%%%%%%%%%%%%%%%%%%%
%%%%%%%%%%%%%%%%%%%%%%%%%%%%%%%%%%%%%%

\section{Agradecimientos}

Esta investigación fue financiada por ANID (ex CONICYT) FONDECYT Regular 1190442 de Chile. Agradezco a Barry Mazur por comentarios en una versión previa de esta nota, y a Cecília Salgado por responder varias dudas.

%%%%%%%%%%%%%%%%%%%%%%%%%%%%%%%%%%%%%%
%%%%%%%%%%%%%%%%%%%%%%%%%%%%%%%%%%%%%%
%%%%%%%%%%%%%%%%%%%%%%%%%%%%%%%%%%%%%%


\begin{thebibliography}{9}         



\bibitem[CPZ05]{CPZ} G. Cornelissen, T. Pheidas, K. Zahidi, \emph{Division-ample sets and the Diophantine problem for rings of integers}.  J. Th\'eor. Nombres Bordeaux 17 (2005), no. 3, 727-735. 

\bibitem[DPR61]{DPR} M. Davis, H. Putnam, J. Robinson, \emph{The decision problem for exponential diophantine equations}. Ann. of Math. (2) 74 (1961) 425-436.

\bibitem[De75]{DenefQuad} J. Denef, \emph{Hilbert's tenth problem for quadratic rings}. Proc. Amer. Math. Soc. 48 (1975), 214-220. 

\bibitem[De80]{Denef}  J. Denef, \emph{Diophantine sets over algebraic integer rings. II.} Trans. Amer. Math. Soc. 257 (1980), no. 1, 227-236.

\bibitem[DL78]{DenLip}  J. Denef, L. Lipshitz, \emph{Diophantine sets over some rings of algebraic integers}. J. London Math. Soc. (2) 18 (1978), no. 3, 385-391.

\bibitem[GP20]{GFP} N. Garcia-Fritz, H. Pasten, \emph{Towards Hilbert's tenth problem for rings of integers through Iwasawa theory and Heegner points}. Math. Ann. 377 (2020), no. 3-4, 989-1013. 

\bibitem[KLS22]{KLS} D. Kundu, A. Lei, F. Sprung, \emph{Studying Hilbert's $10^{th}$ problem via explicit elliptic curves}. Preprint (2022), arXiv:2207.07021 


\bibitem[LN59]{LangNeron} S. Lang, A. N\'eron, \emph{Rational points of abelian varieties over function fields}. Amer. J. Math. 81 (1959), 95-118. 

\bibitem[Ma70]{Matiyasevich} Y. Matiyasevich, \emph{The Diophantineness of enumerable sets}. (Russian) Dokl. Akad. Nauk SSSR 191 (1970) 279-282.

\bibitem[MR10]{MazRubH10} B. Mazur, K. Rubin, \emph{Ranks of twists of elliptic curves and Hilbert's tenth problem}. Invent. Math. 181 (2010), no. 3, 541-575.

\bibitem[MR18]{MazRubDS} B. Mazur, K. Rubin, \emph{Diophantine stability}. With an appendix by Michael Larsen. Amer. J. Math. 140 (2018), no. 3, 571-616.

\bibitem[MRS22]{MRS} B. Mazur, K. Rubin, A. Shlapentokh, \emph{Existential definability and Diophantine stability}. Por aparecer (2022).

\bibitem[MP18]{MurtyPasten} M. R. Murty, H. Pasten, \emph{Elliptic curves, L-functions, and Hilbert's tenth problem}. J. Number Theory 182 (2018), 1-18. 

\bibitem[Ne52]{NeronSp} A. N\'eron, \emph{Probl\`emes arithm\'etiques et g\'eom\'etriques rattach\'es \`a la notion de rang d'une courbe alg\'ebrique dans un corps}. Bull. Soc. Math. France 80 (1952), 101-166.

\bibitem[Ph88]{Pheidas} T. Pheidas, \emph{Hilbert's tenth problem for a class of rings of algebraic integers}. Proc. Amer. Math. Soc. 104 (1988), no. 2, 611-620.

\bibitem[Po02]{Poonen} B. Poonen, \emph{Using elliptic curves of rank one towards the undecidability of Hilbert's tenth problem over rings of algebraic integers}. Algorithmic number theory (Sydney, 2002), 33-42, Lecture Notes in Comput. Sci., 2369, Springer, Berlin, 2002.

\bibitem[Ra22]{Ray} A. Ray, \emph{Remarks on Hilbert's tenth problem and the Iwasawa theory of elliptic curves}. To appear in Bulletin of the Australian Mathematical Society (2022).

\bibitem[Sa12]{Salgado} C. Salgado, \emph{On the rank of the fibers of rational elliptic surfaces}. Algebra Number Theory 6 (2012), no. 7, 1289-1314. 

\bibitem[Sc94]{Schwartz} C. Schwartz, \emph{An elliptic surface of Mordell-Weil rank 8 over the rational numbers}. J. Th\'eor. Nombres Bordeaux 6 (1994), no. 1, 1-8. 

\bibitem[SS89]{ShaShl} H. Shapiro, A. Shlapentokh, \emph{Diophantine relationships between algebraic number fields}. Comm. Pure Appl. Math. 42 (1989), no. 8, 1113-1122.

\bibitem[Sh89]{ShlH10} A. Shlapentokh, \emph{Extension of Hilbert's tenth problem to some algebraic number fields}. Comm. Pure Appl. Math. 42 (1989), no. 7, 939-962.

\bibitem[Sh08]{ShlE} A. Shlapentokh, \emph{Elliptic curves retaining their rank in finite extensions and Hilbert's tenth problem for rings of algebraic numbers}. Trans. Amer. Math. Soc. 360 (2008), no. 7, 3541-3555. 

\bibitem[Si83]{SilvermanSp} J. Silverman, \emph{Heights and the specialization map for families of abelian varieties}. J. Reine Angew. Math. 342 (1983), 197-211.


\bibitem[Vi89]{Videla} C. Videla, \emph{Sobre el d\'ecimo problema de Hilbert}. Atas da Xa Escola de Algebra, Vitoria, ES, Brasil. Colecao Atas 16 Sociedade Brasileira de Matematica (1989), 95-108.

\end{thebibliography}
\end{document}